\documentclass[a4,12pt,leqno]{amsart}
\usepackage{amsmath}
\usepackage{amsfonts}
\usepackage{amssymb}
\usepackage{mathrsfs}

\newcommand{\R}{\mathbb{R}}

\swapnumbers \theoremstyle{plain}
\newtheorem{thm}[equation]{Theorem}
\newtheorem{lemma}[equation]{Lemma}

\theoremstyle{definition}

\theoremstyle{remark}
\newtheorem{rem}[equation]{Remark}

\numberwithin{equation}{section}

\title{Lipschitz equivalence of subsets of self-conformal sets}
\author{Marta Llorente}
\address{Marta Llorente: Departamento de
 An\'{a}lisis Econ\'{o}mico: Econom\'{i}a Cuantitativa\\Universidad Aut\`{o}noma de Madrid, Campus de 
Cantoblanco 28049 Madrid \\  Spain\\ }
\email{m.llorente@uam.es }
\thanks{Part of this job has been done while MLL was visiting the University of Helsinki and supported by the 
Academy of Finland. MLL was also supported by the Ministerio de Educaci´{o}n y Ciencia, research project 
MTM2006-02372.} 
\subjclass[2000]{Primary 28A75,28A80}

\author{Pertti Mattila}
\address{Pertti Mattila: Department of Mathematics and Statistics\\FI-00014 University of Helsinki\\ Finland\\ }
\email{pertti.mattila@helsinki.fi}
\thanks{PM was supported by the Academy of Finland.} 
\subjclass[2000]{Primary 28A75, 28A80}
\keywords{Bilipschitz maps, self-similar sets, self-conformal sets}
\date{}

\begin{document}
\begin{abstract} We give sufficient conditions to guarantee that if two self-conformal sets $E$
and $F$ have Lipschitz equivalent subsets of positive measure, then there is a bilipschitz map
of $E$ into, or onto, $F$.
\end{abstract}

\maketitle

\section{Introduction}

In this note we shall consider the following question: suppose
that $E$ and $F$ are self-conformal (see below for terminology)
subsets of $\R^n$ of the same Hausdorff dimension $s$. If there
are measurable subsets $E'\subset E$ and $F'\subset F$ of positive
$s$-dimensional Hausdorff measure which are Lipschitz equivalent,
are then also $E$ and $F$ Lipschitz equivalent? By Lipschitz equivalence
we mean that there is a bilipschitz map of $E$ onto $F$. We shall prove that
this is true for some Cantor type sets, more precisely, when $E$ and
$F$ satisfy the strong separation condition and one of these sets is
generated by two maps. If $E$ and $F$ satisfy the open
set condition we shall show that there is a bilipschitz map of $E$
into $F$, but not necessarily onto.

Lipschitz equivalence of self-similar and self-conformal sets has
been considered in \cite{FM}, \cite{RRX} and \cite{RRY}, and a general study
of sets having many Lipschitz equivalent subsets can be found in \cite{DS}. In
particular, Falconer and Marsh gave necessary algebraic conditions
on the similarity ratios of the generating maps in order that two
sets satisfying strong separation condition could be Lipschitz
equivalent. These imply, for example, that many self similar subsets
of Hausdorff dimension $\log2/\log3$ are not Lipschitz
equivalent with the classical $1/3$-Cantor set. Hence by our result
in such a case neither are any of their subsets of positive measure.
In \cite{MS} it was shown for the much larger class of Ahlfors-David
regular sets $E\subset\R^n$ and $F\subset\R^n$ of dimensions $s<t<1$,
respectively, that $E$ is Lipschitz equivalent to some subset of $F$.
Combining the results of Falconer and Marsh and the results of this paper
we see that this cannot be extended to the case $s=t$.

\section{Preliminaries}

We shall denote the closed ball with center $x$ and radius $r$ by
$B(x,r)$. The diameter of a set $A$ is denoted by $d(A)$. A map
$h:A\to B, A\subset\R^n, B\subset\R^p$, is said to be bilipschitz,
or $L$-bilipschitz, if there is $L<\infty$ such that
$$|x-y|/L\leq|h(x)-h(y)|\leq L|x-y|\ \text{for all}\ x,y\in A.$$
The smallest such a constant $L$ is denoted by bilip$(h)$. Note that we
don't require $h$ to be onto.

We shall make use of the following simple lemma:

\begin{lemma}\label{le1} Let $A_k\subset A\subset\R^n, B_k\subset B\subset\R^p,$ be compact and
$h_k:A_k\to B_k,k=1,2,\dots,$ be such that for some $L,0<L<\infty$,
$$|x-y|/L\leq|h_k(x)-h_k(y)|\leq L|x-y|\ \text{for all}\ x,y\in A_k, k=1,2,\dots.$$
If for every $x\in A$ there are $x_k\in A_k$ with $x_k\to x$, then
there is an $L$-bilipschitz map $h:A\to B$. If also $h_k(A_k)=B_k$ and for every
$y\in B$ there are $y_k\in B_k$ with $y_k\to y$, then $h(A)=B$.
\end{lemma}

\begin{proof} We can extend the maps $h_k$ to $L$-Lipschitz maps $\R^n\to\R^p$; we shall denote by
 $h_k$ also the extended maps. By the Arzela-Ascoli theorem the
 sequence $(h_k)$ has a subsequence which converges uniformly on
 compact subsets of $\R^n$ to an $L$-Lipschitz map $h$. It is easy
 to check that $h|A:A\to B$ is $L$-bilipschitz, and also the last
 claim is simple.
 \end{proof}

 We denote by $\mathcal H^s$ the $s$-dimensional Hausdorff measure.
 We shall use the fact (see, e.g. Theorem 6.2 in \cite{M}) that for any $\mathcal H^s$ measurable sets
 $A\subset E\subset\R^n$ with $\mathcal H^s(E)<\infty$,
 $\mathcal H^s$ almost all points $x\in A$ are density points of $A$ with respect to $E$ in the sense that
 \begin{align}
 \lim_{r \to 0}r^{-s}\mathcal H^s(B(x,r)\cap E \setminus A)=0.
 \end{align}

 We shall consider a conformal iterated function system $\{f_1,\dots,f_N\}$ in $\R^n$ following the scheme of \cite{MU}.
By this we mean that
 $N\geq2$ and there is an
 open connected set $V\subset\R^n$ such that each $f_i:V\to V$ is an injective  conformal
contraction;
\begin{equation}
\label{eq0}
|f_i(x)-f_i(y)|\leq L_0<1\ \text{for all}\ x,y\in V, i=1,\dots,N,
\end{equation}
of class $C^{1+\gamma}$
 for some fixed $\gamma>0$, that is, the partial derivatives are H\"older continuous with exponent $\gamma$. Of
 course, the last condition is only needed when $n=1$, since the conformal maps in higher dimensions are
 $C^{\infty}$. We shall also assume that there are positive constants $c_1$ and $C_1$ such that $0<c_1<C_1<1$ and
 $$c_1\leq||Df_i(x)||\leq C_1\ \text{for all}\ x\in V, i=1,\dots,N.$$
 Here $||\cdot||$ is the operator norm of a linear map.
 Then there is a unique compact invariant set $E$ such that (see \cite{H})
 $$E=\bigcup_{i=1}^Nf_i(E).$$

 We shall use the following notation.
 Let $\mathcal{N}=\{1,2,...,n\}$ and
 $$\mathcal{N}^{k}=\{\mathbf{i}=(i_{1},...,i_{k}):i_{j}\in \mathcal{N}\ \forall \ j=1,...k\}.$$
 For $\mathbf{i}=(i_{1},...,i_{k})\in\mathcal{N}^{k}$ let
 \begin{eqnarray*}
 &f_\mathbf{i}=f_{i_{1}}\circ f_{i_{2}}\circ
 ...\circ f_{i_{k}},\\
 &E_{\mathbf{i}}=f_{\mathbf{i}}(E),\\
 &d_{\mathbf{i}}=d(E_{\mathbf{i}}).
 \end{eqnarray*}
 Then for every $k=1,2,\dots,$
 $$E=\bigcup_{\mathbf{i}\in\mathcal{N}^{k}}f_{\mathbf{i}}(E).$$

 We shall assume that the system $\{f_i\}$ satisfies the open set
 condition, that is, there is a non-empty bounded open set $O\subset V$
 such that the closure of $O$ is contained in $V$ and the sets $f_i(O)$ are disjoint subsets of $O$.
Then the following bounded distortion property holds, see \cite{MU}, Remark 2.3: there
 is constant $K$ such that for $\mathbf{i}\in\mathcal{N}^{k}$,
 $$||Df_{\mathbf{i}}(x)||\leq K||Df_{\mathbf{i}}(y)||\ \text{for all}\ x,y\in V.$$

 Let $s$ be the Hausdorff dimension of $E$. The open set condition
 implies (and it is in fact equivalent to) that $0<\mathcal
 H^s(E)<\infty$, see \cite{PRSS}. From the bounded distortion
 property one can conclude that there exist
 positive constants $c<1, C$ and $R$ such that for all
 $\mathbf{i}=~(i_1,\dots,i_k)\in\mathcal{N}^{k}$ and $\mathbf{i}
 i=(i_1,\dots,i_k,i)\in\mathcal{N}^{k+1}$,
 \begin{equation}\label{eq1}
 cd_{\mathbf{i}} \leq||Df_{\mathbf{i}}(x)||\leq Cd_{\mathbf{i}}\
 \text{for}\ x\in V,
 \end{equation}
 \begin{equation}\label{eq4'}
 B(f_{\mathbf{i}}(x),cd_{\mathbf{i}}r)\subset f_{\mathbf{i}}(B(x,r))\
 \text{for}\ x\in E, 0<r<R,
 \end{equation}
 \begin{equation}\label{eq4}
 cd_{\mathbf{i}}|x-y| \leq|f_{\mathbf{i}}(x)-f_{\mathbf{i}}(y)|\leq
 Cd_{\mathbf{i}}|x-y|\ \text{for}\ x,y\in E,
 \end{equation}
\begin{equation}\label{eq2}
d_{\mathbf{i} }\leq Cd_{\mathbf{i} i},
 \end{equation}
\begin{equation}\label{eq2'}
d_{\mathbf{i} }\leq L_0^kd(E), d_{\mathbf{i}}\to0\ \text{when}\ \mathbf{i}
 \in\mathcal{N}^k, k\to\infty,
 \end{equation}
 \begin{equation}\label{eq3}
 cr^s \leq\mathcal H^s(B(x,r)\cap E)\leq Cr^s\
 \text{for}\ x\in E, 0<r<R.
 \end{equation}

Here (\ref{eq1}) and (\ref{eq4'}) are proven in Section 2 of \cite{MU},
(\ref{eq4}) follows easily from  (\ref{eq1}), (\ref{eq2}) follows from (\ref{eq4}),
(\ref{eq2'}) follows from  (\ref{eq0}), and (\ref{eq3}) is proven
in Lemma 3.14 of \cite{MU} (where a measure $m$ is used instead of $\mathcal H^s$,
but this is equivalent).

 We shall assume all the time that $f_{\mathbf{i}}, E,
 E_{\mathbf{i}}, d_{\mathbf{i}}, s, c$ and $C$ are as above. We shall
 also consider another conformal iterated function system\newline
 $\{g_1,\dots,g_P\}$ in $\R^p$ and use the corresponding notation
 $g_{\mathbf{i}}, F, F_{\mathbf{i}}, e_{\mathbf{i}},s, c$ and $C$; in
 particular we assume that $E$ and $F$ have the same Hausdorff dimension
 $s$ and we choose the constants $c$ and $C$ so that they match
both systems.

 \section{Open set condition}

 We shall use the result of Peres, Rams, Simon and Solomyak from \cite{PRSS}
 (proven first by Schief in \cite{S} for self-similar sets) according to which
 the open set condition is equivalent to the strong open set condition:
 for some open set $O$ as in the open set condition $E\cap O\not=\emptyset$.
 Both are also equivalent with $0<\mathcal H^s(E)<\infty$.

 It is easy to see that if $O$ is as in the open set condition,
 then $E\subset\bar{O}$, and so for all $\mathbf{i}\in\mathcal N^k,
 E_{\mathbf{i}}\subset f_{\mathbf{i}}(\bar{O})$. Using the strong
 open set condition, choose $x_0\in E\cap O$ and $r_0, 0<r_0<1,$ such that
 $B(x_0,r_0)\subset O$. Then for all $\mathbf{i}\in\mathcal N^k$ by
 (\ref{eq4'}), $B(f_{\mathbf{i}}(x_0),cd_{\mathbf{i}}r_0)\subset
 f_{\mathbf{i}}(B(x_0,r_0))\subset f_{\mathbf{i}}(O)$, from which
 it follows that
 \begin{equation}\label{eq5}
 (E\setminus E_{\mathbf{i}})\cap
 B(f_{\mathbf{i}}(x_0),cd_{\mathbf{i}}r_0)=\emptyset.
 \end{equation}

 \begin{lemma}\label{le2} Suppose that the system $\{f_1,\dots,f_N\}$ satisfies the open set condition.
 Let $b=cr_0/2$ where
 $r_0$ is as above. Then for $\mathcal H^s$ almost all $x\in E$ there are
 $\mathbf{i}_k, k=1,2,\dots$, such that $x\in E_{\mathbf{i}_k},
 d_{\mathbf{i}_k}\to0$ and $(E\setminus E_{\mathbf{i}_k})\cap
 B(x,bd_{\mathbf{i}_k})=\emptyset$ for all $k=1,2,\dots$.
 \end{lemma}

 \begin{proof}
 Let $A_m, m=1,2,\dots,$ be the set of $x\in E$ such that
 $(E\setminus E_{\mathbf{i}})\cap
 B(x,bd_{\mathbf{i}})\not=\emptyset$ whenever $x\in E_{\mathbf{i}}$
 and $d_{\mathbf{i}}<1/m$. We shall show that $\mathcal H^s(A_m)=0$
 which implies the lemma.

 Let $x\in A_m$. If $x\in E_{\mathbf{i}}$ and $d_{\mathbf{i}}<1/m$,
 then by (\ref{eq5}) and the choice of $b$,
 $$(E\setminus E_{\mathbf{i}})\cap B(f_{\mathbf{i}}(x_0),2bd_{\mathbf{i}})=\emptyset,$$
 and so
 $$(E\setminus E_{\mathbf{i}})\cap B(y,bd_{\mathbf{i}})=\emptyset\ \text{for all}\ y\in
 B(f_{\mathbf{i}}(x_0),bd_{\mathbf{i}}),$$
 whence
 $$B(f_{\mathbf{i}}(x_0),bd_{\mathbf{i}})\subset B(x,(1+b)d_{\mathbf{i}})\setminus A_m.$$
 Therefore by (\ref{eq3}),
 $$\mathcal H^s(B(x,(1+b)d_{\mathbf{i}})\cap E\setminus A_m)\geq c(bd_{\mathbf{i}})^s.$$
 It follows that $x$ cannot be a density point
 of $A_m$ and proves that $\mathcal H^s(A_m)=0$.
 \end{proof}

 \begin{thm} \label{thm1}
 Suppose that the systems $\{f_1,\dots,f_N\}$ and
 $\{g_1,\dots,g_P\}$ satisfy the open set condition. If there are
 an $\mathcal H^s$ measurable subset $E'$ of $E$ with $\mathcal
 H^s(E')>0$ and a bilipschitz map $h:E'\to F$, then there exists a
 bilipschitz map $\tilde h:E\to F$.
 \end{thm}

 \begin{proof} We may assume that $E'$ is compact.
 Let $x\in E'$ be a density point of $E'$ such that also $y=h(x)$
 is a density point of $F'=h(E')$ and that, using Lemma \ref{le2},
 there are $b>0$ and $\mathbf{j}_k$ such that $y\in
 F_{\mathbf{j}_k}, e_{\mathbf{j}_k}\to 0$ and $(F\setminus
 F_{\mathbf{j}_k})\cap B(y,b e_{\mathbf{j}_k})=\emptyset$ for all
 $k=1,2,\dots$. Let $L$ be the bilipschitz constant of $h$.
 For $k=1,2,\dots,$ let $\mathbf{i}_k$ be a multi-index
 of shortest length such that $x\in E_{\mathbf{i}_k}$ and  $Ld_{\mathbf{i}_k}\leq be_
 {\mathbf{j}_k}$. Then
 $$h(E'\cap E_{\mathbf{i}_k})\subset B(y,Ld_{\mathbf{i}_k})\cap F'\subset
 B(y,be_{\mathbf{j}_k})\cap F'\cap F_{\mathbf{j}_k},$$
 and, by the minimality of $\mathbf{i}_k$ and (\ref{eq2}), if $\mathbf{i}_k=\mathbf{i}i$, then
 $$d_{\mathbf{i}_k}\geq C^{-1}d_{\mathbf{i}}\geq (CL)^{-1}be_{\mathbf{j}_k}.$$
 Denote
 $$h_k=g_{\mathbf{j}_k}^{-1}\circ h\circ f_{\mathbf{i}_k}:A_k:=f_{\mathbf{i}_k}^{-1}(E'\cap E_{\mathbf{i}_k})\to F.$$
 Then by (\ref{eq4}) $h_k$ is bilipschitz with bilip$(h_k)\leq L'$ with $L'$
 independent of $k$. To complete the proof we shall check that the
 condition of Lemma~\ref{le1} holds for $A_k$ and $A=E$.

 Suppose it doesn't. Then there are $a\in E$ and $r, 0<r<R,$ such that
 for some subsequence of $(\mathbf{i}_k)$, which we assume to be
 the full sequence, we have $B(a,r)\cap
 f_{\mathbf{i}_k}^{-1}(E'\cap E_{\mathbf{i}_k})=\emptyset$. Then by
 (\ref{eq4'})
 $$B(f_{\mathbf{i}_k}(a),cd_{\mathbf{i}_k}r)\cap E_{\mathbf{i}_k}\subset
 f_{\mathbf{i}_k}(B(a,r))\cap E_{\mathbf{i}_k}\subset B(x,d_{\mathbf{i}_k})\cap (E\setminus E').$$
 This gives by (\ref{eq3}) that
 $$\mathcal H^s(B(x,d_{\mathbf{i}_k})\cap(E\setminus E'))\geq c^{s+1}(d_{\mathbf{i}_k}r)^s.$$
 This contradicts the fact that $x$ is a density point and proves the theorem.
 \end{proof}
 \begin{rem} We would like to thank Tamas Keleti for the following
 observation: in  Theorem \ref{thm1} we cannot always get the map
 $\tilde{h}$ to be onto. To see this, take in Theorem \ref{thm1}
 $E'=E \subset \mathbb{R}$ to be a self-similar set satisfying the open set condition
 with positive Lebesgue measure and $E$ not being an interval, $F$ a compact
 interval containing $E$ and $h=id$, the identity map. That such a set $E$ exists can be seen
 for example from \cite{B}.

 Also note that, as a straightforward consequence of Theorem \ref{thm1},
 we find that a self-similar set in $\R^n$ with open set condition and positive Lebesgue
 measure must have a non-empty interior. A simple direct proof for this fact can be found in \cite{S}.

 \end{rem}
 \section{Strong separation condition}

 We say that the strong separation condition holds if the sets
 $f_i(E),i=1,\dots,N$, are disjoint. Then we can choose the
 constant $c$ so that, in addition to the previous properties,
 \begin{equation}\label{eq6}
 \text{dist}(E_{\mathbf{i} i},E_{\mathbf{i} j})\geq cd_{\mathbf{i}}\
 \text{for}\ i\not=j.
 \end{equation}

 \begin{thm} \label{thm 2}
 Suppose that $p=2$ and the systems $\{f_1,\dots,f_N\}$ and
 $\{g_1,g_2\}$ satisfy the strong separation condition. If
 there are an $\mathcal H^s$ measurable subset $E'$ of $E$ with
 $\mathcal H^s(E')>0$ and a bilipschitz map $h:~E'\to F$, then there
 exists a bilipschitz map $\tilde h:E\to F$ with $\tilde h(E)=F$.
 \end{thm}
 \begin{proof} By Theorem~\ref{thm1}, we may assume that $E'=E$.
 Let $x\in E$ be such that $y=h(x)$
 is a density point of $F'=h(E)$. For every $k\in \mathbb{N}$,
 choose $E_{\mathbf{i}_k}$ such that $x\in E_{\mathbf{i}_k}$ and
 $d_{\mathbf{i}_k}\to0$. Then there are sets
 $F_{k,l}=F_{\mathbf{j}_{k(l)}}$ and the corresponding maps
 $g_{k,l}=g_{\mathbf{j}_{k(l)}},l=1,\dots m_k$, such that
 $d(F_{k,l})\geq c_1d_{\mathbf{i}_k}$ and $m_k\leq m$, where
 $c_1>0$ and $m$ are independent of $k,  F'\cap
 F_{k,l}\not=\emptyset$ and
 $$h( E_{\mathbf{i}_k})=\bigcup_{l=1}^{m_k}F'\cap F_{k,l}.$$
 This is essentially Lemma 3.2 in \cite{FM} but we give a quick
 proof. By (\ref{eq6}) dist$(E_{\mathbf{i}_k},E_{\lambda})\geq
 cd_{\mathbf{i}_k}$ whenever $E_{\mathbf{i}_k}\cap
 E_{\lambda}=\emptyset$. If for such $\lambda$ and for some
 $\mathbf{j}$, $h(E_{\mathbf{i}_k})\cap
 F_{\mathbf{j}}\not=\emptyset$ and $h(E_{\lambda})\cap
 F_{\mathbf{j}}\not=\emptyset$, then $d(F_{\mathbf{j}})\geq$
 dist$(h(E_{\mathbf{i}_k}),h(E_{\lambda}))\geq
 (c/L)d_{\mathbf{i}_k}$ where $L$ is the bilipschitz constant of $h$.
 Therefore we can take as $F_{k,l}$ all the maximal sets
 $F_{\mathbf{j}}$ such that $h(E_{\mathbf{i}_k})\cap
 F_{\mathbf{j}}\not=\emptyset$ and
 $d(F_{\mathbf{j}})<(c/L)d_{\mathbf{i}_k}$. Denoting $d_{k,l}=d(F_{k,l})$ we have then by (\ref{eq2}) (as $c/L\leq1$),
 \begin{equation}\label{eq7}
 c(LC)^{-1}d_{\mathbf{i}_k}\leq d_{k,l}\leq d_{\mathbf{i}_k}.
 \end{equation}

 Since $p=2$, we can choose disjoint sets $\tilde{F}_{k,i}=F_{\mathbf{j}_{l(i)}}, i=1,\dots,m_k$, such that
 $d(\tilde{F}_{k,i})\geq c_2$, with $c_2>0$ independent of $k$, and
 $F=\cup_{i=1}^{m_k}\tilde{F}_{k,i}$. Let $\tilde{g}_{k,i}$ be
 the corresponding maps. Note that the maps $\tilde{g}_{k,i}$ are selected from a fixed finite family,
 so their bilipschitz constants have an upper bound independent of $k$. Define
 $$h_k: E\to F$$
 by setting
 $$h_k(x)=\tilde{g}_{k,i}(g_{k,i}^{-1}(h(f_{\mathbf{i}_k}(x))))\ \text{if}\ x\in f_{\mathbf{i}_k}^{-1}(h^{-1}(F'\cap F_{k,i})).$$
 Then bilip$(h_k)\leq L'$ where $L'$ is independent of $k$.
 Namely, the case when $x$, $y\in f_{\mathbf{i}_k}^{-1}(h^{-1}(F'\cap F_{k,i}))$,
 follows by composition. Furthermore, the strong separation
 condition provides us with constants $M_1, M_2>0$ such that
 $$\textrm{dist}(\tilde{F}_{k,i},\tilde{F}_{k,j})\ge M_1$$
 and
 $$\textrm{dist}(F_{k,i},F_{k,j})\ge M_2 d_{\mathbf{i}_k}$$
for all $i\neq j, i,j=1,...m_{k}$. This takes care of the case $x
\in f_{\mathbf{i}_k}^{-1}(h^{-1}(F'\cap~F_{k,i}))$, $y\in
f_{\mathbf{i}_k}^{-1}(h^{-1}(F'\cap F_{k,j}))$ with $i\neq j$.

 We still need to check that the sets
 $$h_k(E)=\bigcup_{i=1}^{m_k}\tilde{g}_{k,i}(g_{k,i}^{-1}(F'\cap F_{k,i}))$$
 and $F$ satisfy the condition for $B_k$ and $B$ in Lemma~\ref{le1}. Suppose this
 is not so. Then there are $a\in F$ and $r, 0<r<R,$ such that $B(a,r)\cap
 h_k(E)=\emptyset$ for some arbitrarily large $k$. For such a
 $k$, $a$ belongs to some $\tilde F_{k,i_0}$. For some $c_3>0$ independent of $k$,
 $B(\tilde g^{-1}_{k,i_0}(a),c_3r)\subset \tilde g^{-1}_{k,i_0}(B(a,r))$ and so, with
 $b=g_{k,i_0}(\tilde g^{-1}_{k,i_0}(a))$ by (\ref{eq4'}),
 $$B(b,c_3crd_{k,i_0})\subset g_{k,i_0}(\tilde g^{-1}_{k,i_0}(B(a,r)))\subset\R^n\setminus F'.$$
 Since $y\in h(E_{\mathbf{i}_k}), b\in F_{k,i_0}, F_{k,i_0}\cap h(E_{\mathbf{i}_k})\not=\emptyset$,
 and $d_{k,i_0}\leq d_{\mathbf{i}_k}$, by (\ref{eq7}), we have $d(y,b)\leq(L+1)d_{\mathbf{i}_k}$ and so by (\ref{eq7}),
 $$B(b,c_4d_{\mathbf{i}_k})\subset B(b,c_3crd_{k,i_0})
 \subset B(y,(L+1)d_{\mathbf{i}_k}+c_3rd_{k,i_0})\subset B(y,c_5d_{\mathbf{i}_k})$$
 with $c_4=c_3c^2(LC)^{-1}r, c_5=L+1+c_3r$. Hence by (\ref{eq3}),
 $$\mathcal H^s(B(y,c_5d_{\mathbf{i}_k})\cap(F\setminus F'))\geq \mathcal H^s(B(b,c_4d_{\mathbf{i}_k})\cap F)
 \geq c(c_4d_{\mathbf{i}_k})^s$$
 contradicting the fact that $y$ is a density point of $F'$.

\end{proof}

\begin{rem} We don't know if the condition $p=2$ is needed in Theorem \ref{thm 2}. Clearly the proof
gives for general $p$ that there is a bilipschitz map of $E$ onto $F_0$ where $F_0$ is a finite union of
sets $F_{\mathbf{j}}$. This raises a question: under what conditions is such a union Lipschitz
equivalent with $F$?

Falconer and Marsh proved in \cite{FM} that for self-similar sets $E$ and $F$ satisfying the
strong separation condition the Lipschitz equivalence of $E$ and $F$ implies certain algebraic conditions on
the similarity ratios of the generating maps. Possibly their method could be modified to prove the
same conditions already if $E$ and $F$ have Lipschitz equivalent measurable subsets of positive
measure. However, we could not deduce from this the Lipschitz equivalence of $E$ and $F$ in
general since it is not clear when the necessary conditions of Falconer and Marsh are also
sufficient.
\end{rem}


\begin{thebibliography}{FSS2}

 \bibitem[B]{B}
 \textsc{C. Bandt}: Self-similar sets 5. Integer matrices and fractal tilings of $\R^n$,
 \emph{Proc. Amer. Math. Soc.}, \textbf{112} (1991), 549-562.

 \bibitem[DS]{DS}
 \textsc{G. David} and \textsc{S. Semmes}:
 \emph{Fractured Fractals and Broken Dreams}, Oxford University Press, 1997.

 \bibitem[FM]{FM}
 \textsc{K.J. Falconer} and \textsc{D.T. Marsh}: On the Lipschitz equivalence of Cantor sets, \emph{Mathematika},
 \textbf{39} (1992), 223-233.

 \bibitem[H]{H}
 \textsc{J.E. Hutchinson}: Fractals and self similarity, \emph{Indiana Univ. Math. J.},
 \textbf{30} (1981), 713-747.

 \bibitem[M]{M}
 \textsc{P. Mattila}: \emph{Geometry of Sets and Measures in Euclidean Spaces},
 Cambridge University Press, 1995.


 \bibitem[MS]{MS}
 \textsc{P. Mattila} and \textsc{P. Saaranen}: Ahlfors-David regular sets and bilipschitz maps,
 \emph{Ann. Acad. Sci. Fenn.}, \textbf{34} (2009), 487-502.

 \bibitem[MU]{MU}
 \textsc{R.D. Mauldin} and \textsc{M. Urbanski}: Dimensions and
 measures in infinite iterated function systems. \emph{Proc. London
 Math. Soc.(3)},  \textbf{ 73, no. 1} (1996), 105-154.

 \bibitem[PRSS]{PRSS}
 \textsc{Y. Peres}, \textsc{M. Rams}, \textsc{K. Simon} and
 \textsc{B. Solomyak}: Equivalence of positive Hausdorff measure and the open
 set condition for self-conformal sets,
 \emph{Proc. Amer. Math. Soc} \textbf{129} (2001), 2689-2699.

 \bibitem[RRX]{RRX}
 \textsc{H. Rao}, \textsc{H.-J. Ruan} and \textsc{L.-F. Xi}: Lipschitz equivalence and self-similar sets,
 \emph{C. R. Math. Acad. Sci. Paris} \textbf{342, no. 3} (2006), 191-196.


 \bibitem[RRY]{RRY}
 \textsc{H. Rao}, \textsc{H.-J. Ruan} and \textsc{Y.-M. Yang}: Gap
 sequence, Lipschitz equivalence and box dimension of fractals,
 \emph{Nonlinearity} \textbf{21} (2008), 1339-1347.

 \bibitem[S]{S}
 \textsc{A. Schief}: Separation properties for self-similar sets.
 \emph{Proc. Amer. Math. Soc.} \textbf{122, no. 1}  (1994),
 111-115.
 \end{thebibliography}
 \end{document}